\documentclass{amsart}
\usepackage{xspace,amssymb}
\usepackage[small]{diagrams}

\def\mat#1{\ensuremath{#1}\xspace}
\def\makedef#1#2{\expandafter\gdef\csname #1\endcsname{#2}}
\def\makelet#1#2{\expandafter\let\csname #1\expandafter\endcsname\csname #2\endcsname}
\def\defmath#1{\makelet{temp@#1}{#1}\makedef{#1}{\mat{\csname temp@#1\endcsname}}}

\def\defbb#1{\makedef{c#1}{\mat{\mathbb{#1}}}}		
\def\defcal#1{\makedef{l#1}{\mat{\mathcal{#1}}}}	
\def\bbcal#1{\defbb{#1}\defcal{#1}}
\bbcal A
\bbcal B
\bbcal C
\bbcal D
\bbcal E
\bbcal F
\bbcal G
\bbcal H
\bbcal I
\bbcal J
\bbcal K
\bbcal L
\bbcal M
\bbcal N
\bbcal O
\bbcal P
\bbcal Q
\bbcal R
\bbcal S
\bbcal T
\bbcal U
\bbcal V
\bbcal W
\bbcal X
\bbcal Y
\bbcal Z

\def\al{\mat{\alpha}}
\def\be{\mat{\beta}}

\def\hi{\mat{\chi}}

\def\La{\mat{\Lambda}}

\def\om{\mat{\omega}}
\def\vi{\mat{\varphi}}

\def\te{\mat{\theta}}

\defmath{eta}
\defmath{xi}
\defmath{mu}
\defmath{nu}
\defmath{Phi}
\defmath{Psi}
\defmath{pi}
\defmath{rho}

\def\mrm@#1{\mat{\mathrm{#1}}}

\def\deffrak#1{\makedef{g#1}{\mat{\mathfrak{#1}}}} 
\deffrak b
\deffrak c
\deffrak C
\deffrak d
\deffrak D
\deffrak h
\deffrak j
\deffrak k

\deffrak m
\deffrak n
\deffrak N
\deffrak p
\deffrak P
\deffrak q
\deffrak r
\deffrak s
\deffrak t
\deffrak u
\deffrak v

\def\DMO{\DeclareMathOperator}

\DMO{\Hom}{Hom}
\DMO{\RHom}{RHom}
\DMO{\lHom}{\lH\mathit{om}}
\DMO{\Ext}{Ext}
\DMO{\lExt}{\lE\mathit{xt}}
\DMO{\End}{End}
\DMO{\Aut}{Aut}
\DMO{\Fun}{Fun}
\DMO{\Tor}{Tor}
\DMO{\ext}{ext}

\DMO{\Ob}{Ob}
\DMO{\Mor}{Mor}
\DMO{\im}{im}
\DMO{\coim}{coim}
\DMO{\coker}{coker}
\DMO{\Arr}{Arr}

\DMO{\Id}{Id}
\DMO{\id}{id}
\DMO{\add}{add} 
\DMO{\ind}{ind} 
\DMO{\pro}{pro} 
\DMO{\Map}{Map} %
\DMO{\Iso}{Iso} %
\DMO{\Isom}{Isom}%
\DMO{\Ind}{Ind}
\DMO{\Bun}{Bun}
\DMO{\Higgs}{Higgs}
\DMO{\Hitch}{Hitch}

\DMO{\Loc}{Loc}
\DMO{\Maps}{Maps}
\makeatletter
\@ifpackageloaded{bbm}
{}
{}
\makeatother

\DMO{\Presh}{Presh}

\DMO\coalg{Coalg}
\DMO{\Rep}{Rep}
\DMO{\Cor}{Cor}

\DMO{\Mod}{Mod}
\DMO{\rad}{rad}
\DMO{\soc}{soc}
\DMO{\ann}{ann}
\DMO{\pd}{pd}

\DMO{\Spec}{Spec}
\DMO{\Specm}{Specm}
\DMO{\Max}{Max}
\DMO{\spec}{Spec}
\DMO{\Proj}{Proj}
\DMO{\supp}{supp}
\DMO{\Coh}{Coh}
\DMO{\coh}{coh}
\DMO{\Qcoh}{QCoh}
\DMO{\QCoh}{QCoh}
\DMO{\Pic}{Pic}
\DMO{\Div}{Div}
\DMO{\ch}{ch}
\DMO{\Hilb}{Hilb}
\DMO{\Fitt}{Fitt}
\DMO{\Quot}{Quot}
\def\Gm{\mat{{\mathbb G}_{\mathrm m}}}
\DMO{\Gras}{Gr}
\DMO{\Grass}{Gr}
\DMO{\Flag}{Flag}
\DMO{\Jac}{Jac}

\DMO{\cone}{cone}
\DMO{\Tw}{Tw}
\DMO{\rank}{rk}
\DMO{\rk}{rk}
\DMO{\codim}{codim}
\DMO{\cov}{cov}
\DMO{\sgn}{sgn}
\DMO{\td}{td}
\DMO{\GL}{GL}
\DMO{\PGL}{PGL}
\DMO{\SL}{SL}
\DMO{\SO}{SO}

\DMO\Der{Der}
\DMO\der{Der}
\DMO\coder{Coder}
\DMO{\diag}{diag}
\DMO{\HMod}{HMod} 
\DMO{\ad}{ad}
\DMO{\Ad}{Ad}
\DMO*{\colim}{colim}
\DMO*{\hocolim}{hocolim}
\DMO*{\holim}{holim}
\DMO{\Ho}{Ho}

\DMO{\har}{char}
\DMO{\sk}{sk}
\DMO{\cosk}{cosk}
\DMO{\Gal}{Gal}
\DMO{\tr}{tr}
\DMO{\Tr}{Tr}
\DMO{\Sh}{Sh}
\DMO{\Is}{Is} 
\DMO{\Hol}{Hol} 
\DMO{\Lie}{Lie} 
\DMO{\Res}{Res} 
\DMO{\irr}{irr} %
\DMO{\Irr}{Irr} %
\DMO{\Exp}{Exp} %
\DMO{\Log}{Log} %
\DMO{\Pow}{Pow}
\DMO{\pow}{pow}
\DMO{\mult}{mult} %
\DMO{\height}{ht} %
\DMO{\wt}{wt}
\DMO{\Vect}{Vect}
\DMO{\moda}{mod}
\DMO{\hd}{hd} 
\DMO{\face}{face}
\DMO{\Sym}{Sym}
\DMO{\Com}{Com} 
\DMO{\Eig}{Eig} 
\DMO{\cl}{cl} 
\DMO{\Li}{Li} 
\DMO{\Imxx}{Im} 
\DMO{\Rexx}{Re}

\def\n#1{\mat{\lvert#1\rvert}}

\def\dd{\mat{\partial}}
\def\iso{\simeq}

\def\tl#1{\mat{\tilde{#1}}}

\def\what#1{\mat{\widehat{#1}}}

\def\sb{\subset}

\def\xx{\times}

\def\ms{\backslash} 
\def\pser#1{[\![#1]\!]} 

\def\half{\mat{\frac12}}
\def\oh{\half} 

\def\inv{^{-1}}

\def\st{{\rm{s}}}
\def\sst{{\rm{ss}}}

\def\ang#1{\mat{\left\langle #1\right\rangle}}

\def\set#1{\mat{\{ #1\}}}
\def\sets#1#2{\mat{\{ #1 \mid #2\}}}

\def\emb{\hookrightarrow}
\def\mto{\mapsto}
\def\arr{\futurelet\test\arrtest}
\def\arrtest{\ifx^\test\let\next\arra\else\let\next\arrb\fi\next}
\def\arra^#1{\xrightarrow{#1}} \def\arrb{\to}

\def\arrowsD{
\def\mto{{\:\vrule height .9ex depth -.2ex width .04em\!\!\!\;\ar}}
\def\ar{{\:\vrule depth -.52ex height .60ex width 0.85em\;\!\!\rhla\,}}
\def\arr{\futurelet\test\arrtest}
\def\arrtest{\ifx^\test\let\next\arra\else\let\next\arrb\fi\next}
\def\arra^##1{\rTo^{##1}} \def\arrb{\ar}
\def\emb{\futurelet\test\embtest}
\def\embtest{\ifx^\test\let\next\emba\else\let\next\embb\fi\next}
\def\emba^##1{\rInto^{##1}} \def\embb{{\:\rthooka\!\!\!\ar}}
\newarrow{Eq}=====
\def\rrarr{\pile{\rTo\\ \rTo}}
\def\lrarr{\pile{\rTo\\ \lTo}}   
\newarrow{ShortTo}{}{}-->
}

\def\arrowsDStandard{
\newarrow{TeXto}----{->}
\newarrow{TeXinto}C---{->}
\newarrow{TeXonto}----{->>}
\newarrow{TeXdashto}{}{dash}{}{dash}{->}
\newarrow{Eq}=====
\def\ar{\rightarrow}
\def\emb{\futurelet\test\embtest}
\def\embtest{\ifx^\test\let\next\emba\else\let\next\embb\fi\next}
\def\emba^##1{\rTeXinto^{##1}} \def\embb{\hookrightarrow}
\def\rrarr{\pile{\rTo\\ \rTo}}
\def\lrarr{\pile{\rTo\\ \lTo}}   
}

\def\theorems{
\newcounter{nthr} 
\numberwithin{nthr}{section}
\let\theHnthr\thenthr
\newtheorem{thr}[nthr]{Theorem}
\newtheorem{prp}[nthr]{Proposition}
\newtheorem{lmm}[nthr]{Lemma}
\newtheorem{crl}[nthr]{Corollary}
\newtheorem{clm}[nthr]{Claim}
\newtheorem{conj}{Conjecture}
\newtheorem{rmr}[nthr]{Remark}
\theoremstyle{definition}
\newtheorem{dfn}[nthr]{Definition}
\newtheorem{exm}[nthr]{Example}
\newtheorem{claim}[nthr]{Claim}
\theoremstyle{remark}
}

\newif\ifrem\remtrue

\def\br{\linebreak}
\def\lb#1{\mat{\underline{#1}}} 
\def\over#1#2{\mat{\substack{#1\\#2}}} 

\def\ie{i.e.\ }
\def\eg{e.g.\ }

\def\eprint#1#2{%
\expandafter\ifx\csname eprint@#1\endcsname\relax#1:#2%
\else\def\itemID{#2}\csname eprint@#1\endcsname\fi}

\def\defArchive#1#2{%
\makedef{eprint@#1}{#2}}

\defArchive{arxiv}{\href{http://arXiv.org/abs/\itemID}{{\tt arXiv:\itemID}}}
\defArchive{numdam}{\href{http://www.numdam.org/item?id=\itemID}{Numdam:\itemID}}

\theorems\arrowsDStandard
\begin{document}
\title[]{On the motivic Donaldson-Thomas invariants of quivers with potentials}%
\author{Sergey Mozgovoy}%
\email{mozgovoy@maths.ox.ac.uk}%
\begin{abstract}
We study motivic Donaldson-Thomas invariants for a class of quivers with potentials using the strategy of Behrend, Bryan, and Szendr\H oi \cite{behrend_motivic}. This class includes quivers with potentials arising from consistent brane tilings and quivers with zero potential. Our construction is an alternative to the constructions of Kontsevich and Soibelman \cite{kontsevich_stability,kontsevich_cohomological}. We construct an integration map from the equivariant Hall algebra to the quantum torus and show that our motivic Donaldson-Thomas invariants are images of the natural elements in the equivariant Hall algebra. We show that the inegration map is an algebra homomorphism and use this fact to prove the Harder-Narasimhan relation for the motivic Donaldson-Thomas invariants.
\end{abstract}
\maketitle
\tableofcontents

\section{Introduction}
The goal of this paper is to study the motivic Donaldson-Thomas invariants for some class of quivers with (polynomial) potentials $(Q,W)$ using the approach of Behrend, Bryan, and Szendr\H oi \cite{behrend_motivic}. These invariants are constructed using the motivic vanishing cycles of functions on smooth moduli spaces of stable quiver representations. The function $w:M^\st_{\te}(Q,\al)\to\cC$ in question is the trace of the potential. It was proved in \cite{behrend_motivic} that if $w$ is equivariant with respect to an appropriate torus action, then the motivic vanishing cycle of $w$ can be computed as
$$[\vi_w]=[w\inv(1)]-[w\inv(0)].$$

We will show that under certain conditions on the potential, we can introduce a weight function on the arrows so that the corresponding torus action on the moduli space will satisfy all the required conditions. Therefore the above equation will hold in this situation. Using the right hand side of this equation for the definition of the motivic Donaldson-Thomas invariants and organizing these invariants as elements of the quantum torus, we will show that they can be obtained as images of some natural elements of the equivariant Hall algebra of the quiver $Q$ with respect to an algebra homomorphism (called an integration map) from the equivariant Hall algebra to the quantum torus. Our integration map is closely related to the integration map of Reineke \cite{reineke_harder-narasimhan} from the whole Hall algebra of $Q$ to the quantum torus. These maps coincide in the case of a trivial potential (see \cite{kontsevich_cohomological,mozgovoy_motivic} on the discussion of the motivic Donaldson-Thomas invariants in this case). In fact, the integration map of Reineke is an important ingredient in our construction.

We should stress, that our construction is quite different from the construction of Kontsevich and Soibelman \cite{kontsevich_stability}, where an integration map from the Hall algebra of the category of modules over the Jacobian algebra to the quantum torus was defined. Our approach is probably less natural, because we use the Hall algebra of the category of quiver representations in order to define some invariants of the moduli spaces of modules over the Jacobian algebra. But it does the job -- all the constructions are quite elementary, the algebra homomorphism property of the integration map is almost obvious, and the relations in the equivariant Hall algebra of the quiver (\eg the Harder-Narasimhan relation) can be translated to the relations in the quantum torus, thus giving us relations between motivic Donaldson-Thomas invariants for different stability parameters. In the last part of the paper we will see how our constructions can be generalized to arbitrary potentials, at least over finite fields.

While preparing this paper I was informed by Kentaro Nagao on his related work \cite{nagao_motivic} on the extension of the approach from \cite{behrend_motivic} to more general quivers with potentials.

I would like to thank Tamas Hausel and Markus Reineke for many useful discussions.
The author's research was supported by the EPSRC grant EP/G027110/1.

\def\Ieq{I_{\rm eq}}
\def\eq{H_{\rm eq}}
\def\heq{\what H_{\rm eq}}
\section{Preliminaries}
\subsection{Bilinear forms related to quivers}
\label{bilin}
Let $Q=(Q_0,Q_1)$ be a quiver. We define the Euler-Ringel form to be the bilinear form on $\cZ^{Q_0}$ given by
$$\hi(\al,\be)=\sum_{i\in Q_0}\al_i\be_i-\sum_{a:i\to j}\al_i\be_j,\qquad \al,\be\in\cZ^{Q_0}.$$
We define the skew-symmetric form
$$\ang{\al,\be}=\hi(\al,\be)-\hi(\be,\al)\qquad \al,\be\in\cZ^{Q_0}.$$
We define the Tits form $T(\al)=\hi(\al,\al)$, $\al\in\cZ^{Q_0}$.
 
\subsection{Representations of quivers with potentials}
Let $(Q,W)$ be a quiver with potential and let $A_W=kQ/(\dd W)$ be the corresponding Jacobian algebra over a field $k$.
A representation $M$ of $Q$ over a field $k$ can be represented as
$$M=((M_i)_{i\in Q_0},(M_a)_{a\in Q_1}),$$
where $M_i$ are $k$-vector spaces and $M_a:M_{s(a)}\to M_{t(a)}$ are linear maps (for any arrow $a\in Q_1$, we denote its source by $s(a)$ and denote its target by $t(a)$).
Let $W=\sum a_uu$, where the sum runs over a finite number of cycles $u$ in $Q$. We define
$$w(M)=\sum a_u \tr(M_u),$$
where for any path $u=a_1\dots a_n$, we define $M_u=M_{a_1}\dots M_{a_n}$.
Note that for any exact sequence of representations of $Q$
$$0\to N\to X\to M\to 0$$
we have $w(X)=w(M)+w(N)$. Therefore, we get a map 
$$w:K_0(\Rep(Q,k))\to k.$$

For any $\al\in\cZ^{Q_0}$ we define the space of representations
$$R(Q,\al)=\bigoplus_{a:i\to j}\Hom(k^{\al_i},k^{\al_j}),$$
where the sum runs over all arrows of $Q$. There is a map 
$$w:R(Q,\al)\to k,\qquad M\mto w(M),$$
which is invariant under the action of $\GL_\al(k)=\prod_{i\in Q_0}\GL_{\al_i}(k)$ on $R(Q,\al)$ by conjugation.
The following result is well-known
\begin{lmm}
A representation $M\in R(\al,\cC)$ is in the degeneracy locus of $w$ (\ie $dw(M)=0$) if and only if $M$ is a representation of the Jacobian algebra $A_W$.
\end{lmm}

\subsection{Moduli spaces}
Let $k=\cC$ in this section.
Let $\te\in\cR^{Q_0}$ be some fixed vector. For any $\al\in\cN^{Q_0}\ms\set0$, we define
$$\mu_\te(\al)=\frac{\te\cdot\al}{\sum\al_i}.$$
For any $Q$-representation $M$, we define $\mu_\te(M)=\mu_\te(\lb\dim M)$, where $\lb\dim M=(\dim M_i)_{i\in Q_0}\in\cN^{Q_0}$ is the dimension vector of $M$.
We say that a representation $M$ is semistable (resp.\ stable) if for any $0\ne N\subsetneq M$ we have $\mu_\te(N)\le\mu_\te(M)$ (resp.\ $\mu_\te(N)<\mu_\te(M)$).

Let $M^\sst_\te(Q,\al)$ (resp.\ $M^\st_\te(Q,\al)$) be the moduli space of \te-semistable (resp.\ stable) representations of $Q$ having dimension vector \al.
We denote by $\lM^\sst_\te(Q,\al)$ the stack of \te-semistable representations of $Q$ of dimension \al. The potential map $w:R(Q,\al)\to\cC$ descends to $w:M_\te^\sst(Q,\al)\to\cC$ and $\tl w:\lM_\te^\sst(Q,\al)\to\cC$. Its degeneracy locus on $w:M_\te^\st(Q,\al)\to\cC$ coincides with the moduli space $M_\te^\st(A_W,\al)$ of \te-stable $A_W$-modules having dimension vector \al.

We say that $\te$ is \al-generic if for any $0<\be<\al$ we have $\mu_\te(\be)\ne\mu_\te(\al)$. Then any semistable $Q$-representation of dimension \al is automatically stable.

\subsection{Weights}
Given a map $\wt:Q_1\to\cN$, we define, for any path $u=a_1\dots a_n$,
$\wt(u)=\sum \wt a_i$. We extend the weight function also to $\wt:\cZ^{Q_1}\to\cZ$ by linearity. 

\begin{rmr}
Throughout the paper we will make an assumption that the weight function is positive on all cycles, and moreover $\wt(u)$ is constant on all cycles $u$ having nonzero coefficient in $W$ (we say that $W$ is homogeneous with respect to the weight function and denote the weight of its cycles by $\wt(W)$).
\end{rmr}

\begin{rmr}
Such choice of weight function is always possible for quivers with potentials arising from consistent brane tilings \cite{mozgovoy_noncommutative}. For example, we can choose a weight function corresponding to the perfect matching of the associated bipartite graph (define the weight of an arrow to be equal $1$ if it is in the perfect matching and zero otherwise).
\end{rmr}

\begin{rmr}
\label{primitive condition}
Let $Q_2$ be the set of cycles having nonzero coefficients in $W$. Consider the maps
$\cZ^{Q_2}\to \cZ^{Q_1}$ (sending every cycle to its content) and $\cZ^{Q_2}\to\cZ$ sending every cycle to $1$. We form a cocartesian diagram
\begin{diagram}
\cZ^{Q_2}&\rTo&\cZ^{Q_1}\\
\dTo&&\dTo\\
\cZ&\rTo^\om&\La
\end{diagram}
Our assumption that $W$ is homogeneous with respect to $\wt$ means that $\wt:\cZ^{Q_1}\to\cZ$ can be uniquely factored through $\cZ^{Q_1}\to\La$ so that the composition $\cZ\arr^\om\La\to\cZ$ is the multiplication by $\wt W$.
Dualizing, we get a map of tori $T_\La=\Hom(\La,\cC^*)\to(\cC^*)^{Q_1}$, which induces an action of $T_\La$ on the moduli spaces of $Q$-representations. 

The map
$w:M_\te^\sst(Q,\al)\to\cC$ is $T_\La$-equivariant, where the action on $\cC$ is given by the character $\hi_\om:T_\La\to\cC$ induced by $\om$. In order to apply the results of \cite{behrend_motivic}, we need the character $\hi_\om$ to be primitive. Therefore, we require the map $\om:\cZ\to\La$ to be a split monomorphism. It was shown in \cite{mozgovoy_noncommutative} that this is always the case for potentials arising from consistent brane tilings. Note that the map $\wt:\cZ^{Q_1}\to\cZ$ can be factored through $\wt:\La\to\cZ$ with a composition $\wt\circ\om:\cZ\to\cZ$ being multiplication by $\wt W$. 
\end{rmr}

For any representation $M$ and an element $t\in\Gm$, we define a new representation $tM$ as follows 
$$(tM)_i=M_i,\quad i\in Q_0,\quad (tM)_a=t^{\wt(a)}M_a,\quad a\in Q_1.$$
Then
$$w(tM)=t^{\wt W}w(M).$$
If $M$ is a representation of the jacobian algebra $A_W$ then so also is $tM$.

\section{Framed quiver representations}
The goal of this section is to show that, at least for framed quiver representations, we can define the circle-compact action of $\cC^*$ on the corresponding moduli spaces (the action of $\cC^*$ on a variety $X$ is called circle compact if $X^{\cC^*}$ is compact and there exists the limit $\lim_{t\to0}tx$ for any $x\in X$). As is mentioned in \cite{behrend_motivic}, we don't actually need the compactness of the $\cC^*$-invariant part in order to apply \cite[Prop.~1.11]{behrend_motivic}. Nevertheless we prove our compactness result for completeness.

Let $Q'$ be a new quiver obtained from $Q$ by adding a new vertex $*$ and some arrows from $*$ to $Q_0$ and from $Q_0$ to $*$.
We will study moduli spaces of $Q'$-representations having dimension vector $\al'=(\al,1)$, where $\al\in\cN^{Q_0}$. We call such representations framed $Q$-representations. Let us choose a stability parameter $\te'=(\te,\te_*)\in\cR^{Q_0}\xx\cR$.
We extend the weight function $\wt:Q_1\to\cN$ to a function $\wt:Q'_1\to\cN$ by arbitrary positive integer values. It defines an action of $\cC^*$ on the moduli space $M_{\te'}^\sst(Q',\al')$.

\begin{thr}
Let $\te'\in\cR^{Q'_0}$ and let $\al'=(\al,1)\in\cN^{Q'_0}$.
Assume that $\te'$ is $\al'$-generic. Then the subvariety of $\Gm$-invariant points of $M_{\te'}^\sst(Q',\al')$ is a projective variety.
\end{thr}
\begin{proof}
We follow the strategy of \cite{reineke_localization}.
Consider some $\Gm$-invariant point in\br
$M_{\te'}^\sst(Q',\al')$. It is automatically stable and is represented by some $M\in R(Q',\al')$. There is a natural action of the group $\GL_\al=\prod_{i\in Q_0}\GL_{\al_i}$ on $R(Q',\al')$ by conjugation. By our assumption for any $t\in\Gm$ there exists some $g=(g_i)_{i\in Q_0}$ such that for any arrow $a:i\to j$ in $Q'$
$$(tM)_a=g_j M_a g_i\inv,$$
where we define $g_*=1$. Let $H\sb \Gm\xx\GL_\al$ be the subgroup of all elements $(t,g)$ satisfying this condition. Then $p_1:H\to\Gm$ is surjective. But its kernel is trivial. Indeed, if $(1,g)$ is in the kernel, then $g$ is an automorphism of $M$ (with $g_*=1$). It follows from the stability of $M$, that $g=1$. Consider the composition
$$\psi:p_2\circ p_1:\Gm\to\GL_\al$$
and split it to components $\psi_i:\Gm\to\GL(M_i)$, $i\in Q_0$. We can decompose $M_i$ with respect to the character group $\cZ$ of \Gm
$$M_i=\bigoplus_{n\in\cZ}M_{i,n}, \quad i\in Q_0.$$
We also decompose $M_*=M_{*,0}$. 
One can see that for any arrow $a:i\to j$ we have
$$M_a(M_{i,n})\sb M_{j,n+\wt a}.$$
One can show that conversely, the existence of such grading on $M$ implies that $M$ is fixed by the action of $\Gm$.

We can find some boundary $N$ such that $M_{i,n}=0$
for $i\in Q'_0$, $\n n>N$. For example we can take 
$$N=\n\al m,\qquad \n\al=\sum_{i\in Q'_0}\al_i=\dim M,\quad m=\max_{a\in Q'_1}\wt a.$$
Indeed, if there exists say $n_0>N$ such that $M_{i,n_0}\ne0$ for some $i\in Q'_0$, then there exists $0\le k<k+m<n_0$ such that $M_{i,n}=0$ for all $i\in Q'_0$, $k+1\le n\le k+m$. But then $M$ is a direct sum of two submodules 
$$\bigoplus_{i\in Q_0,n\le k} M_{i,n},\qquad \bigoplus_{i\in Q_0,n>k+m} M_{i,n},$$
where the first submodule is nonzero because $M_{*,0}\ne0$. This contradicts to the stability of $M$.

The grading of $M$ allows us to construct a representation $\what M$ of the following quiver $\what Q$. Its vertices are pairs
$$(i,n),\quad i\in Q'_0,\quad -N\le n\le N$$
and its arrows are pairs 
$$(a,n):(s(a),n)\to(t(a),n+\wt(a)),\quad a\in Q'_1,\quad -N\le n\le n+\wt a\le N.$$
The dimension vector of $\what M$ equals $\what\al\in \cZ^{\what Q_1}$ given by $\what\al_{i,n}=\dim M_{i,n}$. 

It is clear that $\what M$ is stable with respect to the stability condition $\what \te$ defined by $\what\te_{i,n}=\te'_i$. Conversely, given a $\what\te$-stable representation $\what M$ of quiver $\what Q$ of dimension $\what\al$, we can construct a representation $M$ of quiver $Q'$ of dimension $\al'$ which is fixed by $\Gm$. We claim, that $M$ is $\te'$-stable (or, equivalently, $\te'$-semistable). If this is wrong then there exists a destabilizing semistable submodule $N\sb M$ coming from the Harder-Narasimhan filtration. It follows from the uniqueness of the Harder-Narasimhan filtration that $N$ is actually $\Gm$-invariant. Let $U\sb N$ be some stable submodule. For any $t\in\Gm$, the stable representation $tU$ is a submodule of $tN\iso N$. It follows that $tU, t\in\Gm$, form a direct sum in $N$ (because they are simple objects in the category of semistable modules having the same slope as $N$). As $N$ is finite dimensional, the orbit $tU,t\in\Gm$, in the moduli space of stable modules should be finite. But it is an image of a connected group $\Gm$, so the orbit consists of just one element. This implies that $U$ is $\Gm$-invariant. The same argument as earlier for the module $M$ shows that $U$ can be considered as a representation $\what U$ of the quiver $\what Q$. Therefore $\what U$ is a destabilizing submodule of $\what M$ and this contradicts to the stability of $\what M$.

We have shown that the subvariety of $\Gm$-invariant points of $M_{\te'}^\sst(Q',\al')$ can be identified with the moduli space $M^\sst_{\what\te}(\what Q,\what\al)$ (note that $\what\te$ is $\what\al$-generic). Note that the quiver $\what Q$ is acyclic, because any cycle in $\what Q$ would project to a cycle in $Q'$ and all cycles in $Q'$ have positive weights by our assumptions. This implies that $M^\sst_{\what\te}(\what Q,\what\al)$ is projective and the theorem is proved.
\end{proof}

\begin{lmm}
\label{limit}
For any $M\in M^\sst_{\te}(Q,\al)$, there exists the limit $\lim_{t\to0}tM$.
\end{lmm}
\begin{proof}
Let $\te_0=0\in\cR^{Q_0}$. Then $M^\sst_{\te_0}(Q,\al)$ is affine and there exists a proper map $\pi:M^\sst_{\te}(Q,\al)\to M^\sst_{\te_0}(Q,\al)$ (see \cite{king_moduli}). For any $M\in M^\sst_{\te_0}(Q,\al)$ there exists the limit $\lim_{t\to0}tM$ (it exists already in $R(Q,\al)$). Now our statement follows from the properness of $\pi$. 
\end{proof}

\begin{rmr}
It is not true in general, that for any $M\in R^\sst_\te(Q,\al)\sb R(Q,\al)$ there exists the limit $\lim_{t\to0}tM$ in $R^\sst_\te(Q,\al)$. Therefore we don't formulate analogous statement for the moduli stacks of representations. For example, consider the quiver $Q$ with two vertices $1,2$ and one arrow $a:1\to2$. Let $\al=(1,1)$, $\te=(1,0)$ and let the action of $\cC^*$ be given by multiplication. Then $M^\sst_\te(Q,\al)$ consists of one representation $M=[\cC\arr^{1}\cC]$ and the limit $\lim_{t\to0}tM$ exists and coincides with $M$. On the other hand $R^\sst_\te(Q,\al)$ consists of representations $M_s=[\cC\arr^s\cC]$, $s\in\cC^*$, and the limit $\lim_{t\to0}tM_1=[\cC\arr^0\cC]$ is not contained in $R^\sst_\te(Q,\al)$.
\end{rmr}

\section{Motivic Donaldson-Thomas invariants}
Let $\al\in\cN^{Q_0}$ and let $\te\in\cR^{Q_0}$ be \al-generic. Then all representations in $M^\sst_\te(Q,\al)$ are stable and therefore $M^\sst_\te(Q,\al)$ is smooth.
Consider the trace of the potential (also for \te non-\al-generic) for the moduli space and the moduli stack
$$w:M^\sst_{\te}(Q,\al)\to\cC,$$
$$\tl w:\lM^\sst_{\te}(Q,\al)\to\cC.$$
We denote by $[\vi_w]$ the absolute motivic vanishing cycle of $w$ (see \eg \cite{behrend_motivic}). We can write the moduli stack $\lM^\sst_\al(Q,\al)$ as the global quotient stack $[M^\sst_\al(Q,\al)/\cC^*]$ with a trivial action of $\cC^*$.
We refer to \cite{behrend_motivic} for the definition of the motive of an algebraic stack. In particular, we have
$$[\lM^\sst_\al(Q,\al)]=\frac{[M^\sst_\al(Q,\al)]}{[\cC^*]}=\frac{[M^\sst_\al(Q,\al)]}{\cL-1}.$$
To avoid the definition of the motivic vanishing cycle for stacks, we just define $[\vi_{\tl w}]$ to be $\frac{[\vi_w]}{\cL-1}$. Following \cite[Definition 1.13]{behrend_motivic}, we formulate

\begin{dfn}
Let $\al\in\cN^{Q_0}$ and let $\te\in\cR^{Q_0}$ be $\al$-generic.
We define the virtual motive $A_{\al}^{\te}$ of the moduli stack $\lM^\sst_{\te}(A_W,\al)$ by the formula
$$A_{\al}^{\te}=-(-\cL^\oh)^{-\dim\lM^\sst_{\te}(Q,\al)}[\vi_{\tl w}].$$
\end{dfn}

\begin{rmr}
Note that
\begin{equation}
\dim\lM^\sst_{\te}(Q,\al)=\dim M^\sst_{\te}(Q,\al)-1=-T(\al).
\end{equation}
where $T$ is a Tits form of the quiver $Q$.
\end{rmr}

\begin{thr}
We have
$$[\vi_w]=[w\inv(1)]-[w\inv(0)].$$
\end{thr}
\begin{proof}
We are going to apply \cite[Prop.~1.11]{behrend_motivic}.
We have seen, that the moduli space $M^\sst_{\te}(Q,\al)$ admits a $\cC^*$-action such that
for any point $M\in M^\sst_{\te}(Q,\al)$ there exists the limit $\lim_{t\to0}tM$ (see Lemma \ref{limit}). Moreover, by our assumptions (see Remark \ref{primitive condition}), the map $w:M^\sst_{\te}(Q,\al)\to\cC$ is $T_\La$-equivariant, where the action of $T_\La$ on $\cC$ is given by the primitive character $\hi_w:T_\La\to\cC^*$. Thus, we can apply \cite[Prop.~1.11]{behrend_motivic} (it is noted there that the condition on the compactness of the $\cC^*$-invariant part can be dropped) and we get the statement of the theorem. 
\end{proof}

\begin{crl}
We have
$$A_{\al}^{\te}=(-\cL^\oh)^{T(\al)}\frac{[\lM^\sst_{\te}(Q,\al)]-\cL[\tl w\inv(0)]}{1-\cL}.$$
\end{crl}
\begin{proof}
By the previous result
$$[\vi_w]
=[w\inv(1)]-[w\inv(0)]
=\frac{[M^\sst_{\te}(Q,\al)]-\cL[w\inv(0)]}{\cL-1}.$$
We use now
$$[\lM^\sst_{\te}(Q,\al)]=\frac{[M^\sst_{\te}(Q,\al)]}{\cL-1},\qquad 
[\tl w\inv(0)]=\frac{[w\inv(0)]}{\cL-1}.$$
\end{proof}

\begin{dfn}
For a not necessarily $\al$-generic stability parameter $\te$, we define the virtual motive $A_{\al}^{\te}$ of the moduli stack $\lM^\sst_{\te}(A_W,\al)$ to be
\begin{equation}
A_{\al}^{\te}=(-\cL^\oh)^{T(\al)}\frac{[\lM^\sst_{\te}(Q,\al)]-\cL[\tl w\inv(0)]}{1-\cL}.
\label{general def}
\end{equation}
\end{dfn}
For any $\mu\in\cR$, we define the motivic Donaldson-Thomas series (they are elements of the motivic quantum torus, see the next section)
$$A_\mu^\te=\sum_{\mu_\te(\al)=\mu}A_\al^\te x^\al.$$

Note that Equation \eqref{general def} can be interpreted even over finite fields. We will use this fact in order to work with Hall algebras over finite fields. This is done just to make the exposition more clear. The generalization of our results to motivic Hall algebras is straightforward.

\section{Equivariant Hall algebra}
Let $H$ be the Hall algebra of the category of representations of the quiver $Q$ over a finite field $k=\cF_q$ (we use the conventions of \cite{kontsevich_stability} fot the multiplication and this gives an algebra opposite to the usual Ringel-Hall algebra).
The basis of $H$ as a vector space consists of all isomorphism classes of representations of $Q$ over $\cF_q$. Multiplication is given by the rule
$$[N]\circ[M]=\sum_{[X]}F_{MN}^X[X],$$
where
$$F_{MN}^X=\#\sets{U\sb X}{U\iso N,\ X/U\iso M}.$$
Let $\eq\sb H$ be a subalgebra consisting of elements $f=\sum a_M[M]$ such that $a_{tM}=a_M$ for any $t\in k^*$. We call $\eq$ the equivariant Hall algebra. For any $f=\sum a_M[M]\in \eq$, we define $f_0=\sum_{w(M)=0}a_M[M]$. The algebras $H,\eq$ are graded by the dimension of representations. We denote by $\what H,\heq$ the corresponding completions.

Let $A_q$ be the quantum torus of the quiver $Q$. As a vector space it is $$\cQ(q^\oh)\pser{x_1,\dots,x_r},$$
where $r=\# Q_0$. Multiplication is given by
$$x^\al\circ x^\be=(-q^\oh)^{\ang{\al,\be}}x^{\al+\be},$$
where $\ang{-,-}$ is the skew-symmetric form of the quiver $Q$, see Section \ref{bilin}.
It was shown by Markus Reineke \cite{reineke_counting} that there exists an algebra homomorphism
\begin{equation}
I:\what H\to A_q,\qquad [M]\mto\frac{(-q^\oh)^{T(\lb\dim M)}}{\#\Aut M}x^{\lb\dim M}.
\label{eq:int map}
\end{equation}

\begin{rmr}
Similarly, one can define the motivic Hall algebra of representations over $Q$ (see \eg \cite{joyce_configurationsa,kontsevich_stability,bridgeland_introduction}) and the quantum torus over the Grothendieck ring of the category of Chow motives \cite{kontsevich_stability}, where multiplication is given by
$$x^{\al}\circ x^\be=(-\cL^\oh)^{\ang{\al,\be}}x^{\al+\be}.$$
There is an algebra homomorphism from the motivic Hall algebra to the motivic quantum torus similar to the above map. All the statements of this section can be proved in the motivic setting without any additional effort.
\end{rmr}

Let $\tl w:\lM_{\te}^\sst(Q,\al)\to k$ be the trace of the potential.
The invariants
$$A_{\al}^{\te}=(-\cL^\oh)^{T(\al)}\frac{[\lM^\sst_{\te}(Q,\al)]-\cL[\tl w\inv(0)]}{1-\cL}.$$
defined earlier, can be also defined over a finite field $\cF_q$. The point count of the stack $\lM_{\te}^\sst(Q,\al)$ (multiplied by $(-q^\oh)^{T(\al)}$) corresponds to $I(\tl A_{\al}^{\te})$, where
$$\tl A_{\al}^{\te}=\sum_{\over{M\text{ is }\te-sst}{\lb\dim M=\al}}[M]\in \heq.$$
The point count of $\tl w\inv(0)$ (multiplied by $(-q^\oh)^{T(\al)}$) corresponds to $I((\tl A_{\al}^{\te})_0)$
(recall that for $f=\sum a_M[M]$, we define $f_0=\sum_{w(M)=0}a_M[M]$). Therefore, over a finite field $\cF_q$, we define
$$A_{\al}^{\te}x^\al=\frac{I(\tl A_{\al}^{\te})-qI((\tl A_{\al}^{\te})_0)}{1-q}.$$

\begin{prp}
The map $\Ieq:\heq\to A_q$
$$\Ieq(f)=\frac{I(f)-q I(f_0)}{1-q}$$
is an algebra homomorphism.
\end{prp} 
\begin{proof}
For any $f\in \eq$, $t\in k^*$ we have
$$I(f_t)=\frac{I(f)-I(f_0)}{q-1}.$$
The map $w$ is additive with respect to exact sequences.
Therefore
$$(fg)_0=\sum_{t\in k}f_t g_{-t}.$$
Let $F=I(f),F_0=I(f_0),G=I(g),G_0=I(g_0)$. Then
$$I((fg)_0)=\sum_{t\in k}I(f_t)I(g_{-t})
=\frac{(F-F_0)(G-G_0)}{q-1}+F_0G_0.$$
Therefore
\begin{align*}
\Ieq(fg)=&
\frac{I(fg)-qI((fg)_0)}{1-q}\\
=&\frac{(q-1)FG-q(F-F_0)(G-G_0)-q(q-1)F_0G_0}{-(q-1)^2}\\
=&\frac{-FG+q(FG_0+F_0G)-q^2F_0G_0}{-(q-1)^2}\\
=&\frac{(F-qF_0)(G-qG_0)}{(q-1)^2}=\Ieq(f)\Ieq(g).
\end{align*}
\end{proof}

\begin{rmr}
It follows from the above discussion that 
$$A_{\al}^{\te}x^\al=\Ieq(\tl A_{\al}^{\te}).$$
\end{rmr}

\begin{rmr}
Note that if $f=f_0$, then $\Ieq(f)=I(f)$. This implies that the unit element is sent to the unit element by $\Ieq$. This implies also that for a trivial potential we have $\Ieq=I$.
\end{rmr}

\begin{dfn}
For any $\mu\in\cR$, we define
$$\tl A_\mu^{\te}=\sum_{\mu_{\te}(\al)=\mu}\tl A_{\al}^{\te}
=\sum_{\over{M\text{ is }\te-sst}{\lb\dim M=\al}}[M]\in \heq.$$
We define the motivic Donaldson-Thomas series
$$A_\mu^{\te}=\Ieq(\tl A_\mu^{\te})=\sum_{\mu_{\te}(\al)=\mu}A_{\al}^{\te}x^\al.$$
For $\te=0$ and $\mu=0$, we denote $\tl A_\mu^\te$ (resp.\ $A_\mu^\te$ and $A^\te_\al$) just by $\tl A$ (resp.\ $A$ and $A_\al$).
\end{dfn}

It follows from the Harder-Narasimhan filtration for $Q$-representations that, for any $\te\in\cR^{Q_0}$, we have
$$\tl A=\prod^{\leftarrow}_{\mu}\tl A_{\mu}^{\te}\in\heq,$$
where the product is taken in the decreasing order of $\mu\in\cR$.
Applying the associative map $\Ieq$, we obtain an analogous statement in the quantum torus.

\begin{thr}[Harder-Narasimhan relation]
For any $\te\in\cR^{Q_0}$, we have
$$A=\prod^{\leftarrow}_{\mu}A_{\mu}^{\te},$$
where the product is taken in the decreasing order of $\mu\in\cR$.
\end{thr}

This recursion formula can be solved using the approach of Markus Reineke \cite{reineke_harder-narasimhan} (see also \cite[Theorem 3.2]{mozgovoy_poincare}).

\begin{thr}
For any $\te\in\cR^{Q_0}$, we have
$$A^\te_\al=\sum_{(\al_1,\dots,\al_k)}(-1)^{k-1}(-q^\oh)^{\sum_{i<j}\ang{\al_i,\al_j}}\prod_{i=1}^k A_{\al_i}.$$
where the sum runs over all tuples $(\al_1,\dots,\al_k)$ of vectors in $\cN^{Q_0}\ms\set0$ such that $\sum_{i=1}^k\al_i=\al$ and $\mu(\sum_{i=1}^j\al_i)>\mu(\al)$ for any $1\le j<k$.
\end{thr}
\begin{proof}
According to the previous theorem, for any $\al\in\cN^{Q_0}\ms\set0$, we can write
$$A_\al x^\al=\sum_{(\al_1,\dots,\al_k)}(A^\te_{\al_1}x^{\al_1})\circ\dots\circ(A^\te_{\al_k}x^{\al_k}),$$
where the sum runs over all tuples $(\al_1,\dots,\al_k)$ of vectors in $\cN^{Q_0}\ms\set0$ such that $\sum_{i=1}^k\al_i=\al$ and $\mu(\al_1)>\dots>\mu(\al_k)$. Applying \cite[Theorem 3.2]{mozgovoy_poincare} we deduce that
$$A^\te_\al x^\al=\sum_{(\al_1,\dots,\al_k)}(-1)^{k-1}(A_{\al_1}x^{\al_1})\circ\dots\circ(A_{\al_k}x^{\al_k}),$$
where the sum runs over all tuples $(\al_1,\dots,\al_k)$ of vectors in $\cN^{Q_0}\ms\set0$ such that $\sum_{i=1}^k\al_i=\al$ and $\mu(\sum_{i=1}^j\al_i)>\mu(\al)$ for any $1\le j<k$.
The statement of the theorem follows now from the definition of the multiplication in the quantum torus~$A_q$.
\end{proof}

\begin{rmr}
It follows from the above results that if $A_\al$ are rational functions in $q^\oh$, for $\al\in\cN^{Q_0}$, then so are also $A^\te_\al$ for any stability parameter $\te\in\cR^{Q_0}$.
\end{rmr}

\section{Quivers with arbitrary potentials}
As we explained earlier, the results of the previous section can be proved also in the motivic setting. In this section, however, we will work only with quiver representations over finite fields.
Let $\cF_q$ be some finite field.
We extend scalars in the quantum torus and define
$$A_q=\cC\pser{x_1,\dots,x_r}$$
with multiplication $x^\al\circ x^\be=(-q^\oh)^{\ang{\al,\be}}x^{\al+\be}$.

Let $\psi:\cF_q\to\cC^*$ be some non-trivial character.
Define the map 
$$I_\psi:\what H\to A_q,\qquad [M]\mto \psi (w(M))I(M),$$
where the map $I:\what H\to A_q$ was defined in \eqref{eq:int map}.

\begin{lmm}
The map $I_\psi:\what H\to A_q$ is an algebra homomorphism.
\end{lmm}
\begin{proof}
The map $w$ is additive with respect to exact sequences.
Therefore
\begin{align*}
I_\psi([N]\circ [M])=&I_\psi\left(\sum F_{MN}^X X\right)\\
=&\psi w(M)\psi w(N)I\left(\sum F_{MN}^X[X]\right)\\
=&\psi w(M)\psi w(N)I(N)I(M)
=I_\psi(N)I_\psi(M).
\end{align*}
\end{proof}

Recall that in the previous section we have used the weight function $\wt:Q_1\to\cZ$ in order to define the action of $\Gm$ on the quiver representations and to define the equivariant Hall algebra $\eq\sb H$.

\begin{lmm}
For any $f\in \heq$, we have $\Ieq(f)=I_\psi(f)$.
\end{lmm}
\begin{proof}
It is enough to prove the statement for $f=\sum_{t\in\cF_q^*}[tM_0]$, where $w(M_0)\ne0$, and for $f=[M_0]$, where $w(M_0)=0$.
In the first case we have
$$I_\psi(M)
=\sum_{t\in\cF_q^*}\psi w(tM_0)I(M_0)
=\sum_{t\in\cF_q^*}\psi(t)I(M_0)=-I(M_0)=\frac{I(f)}{1-q}=\Ieq(f).$$
In the second case we have
$$I_\psi(f)=I(M_0)=\Ieq(f).$$
\end{proof}
This lemma means that instead of using the homomorphism $\Ieq:\heq\to A_q$ for the definition of Donaldson-Thomas series, we can use the homomorphism $I_\psi:\what H\to A_q$. While the homomorphism $\Ieq$ depends on the weight function $\wt:Q_1\to\cZ$ (more precisely, its domain $\heq$ depends on $\wt$), the homomorphism $I_\psi$ depends only on the character $\psi$ and its domain is the whole Hall algebra $\what H$.
We can use $I_\psi$ to define the Donaldson-Thomas series for arbitrary potentials, as this approach does not require a weight function $\wt:Q_1\to\cZ$ which is compatible with the potential.

\bibliography{../tex/papers}
\bibliographystyle{../tex/hamsplain}

\end{document}